\documentclass[a4paper,12pt]{article}

\usepackage{graphicx} 
\usepackage{amsmath} 
\usepackage{amssymb}  
\usepackage{color}

\usepackage{amsthm}
\usepackage{hyperref} 

\newcommand{\add}[1]{#1}

\newcommand{\RR}{\mathbb{R}}
\newcommand{\NN}{\mathbb{N}}

\theoremstyle{plain}
\newtheorem{theorem}{Theorem}[section]

\newtheorem{proposition}[theorem]{Proposition}

\theoremstyle{definition}

\newtheorem{example}{Example}

\theoremstyle{remark}
\newtheorem*{remark}{Remark}

\title{\bf Inverse optimal control with polynomial optimization}

\begin{document}

\author{Edouard Pauwels$^{1,2}$, Didier Henrion$^{1,2,3}$, Jean-Bernard Lasserre$^{1,2}$}

\footnotetext[1]{CNRS; LAAS; 7 avenue du colonel Roche, F-31400 Toulouse; France.}
\footnotetext[2]{Universit\'e de Toulouse;  LAAS, F-31400 Toulouse, France.}
\footnotetext[3]{Faculty of Electrical Engineering, Czech Technical University in Prague,
Technick\'a 2, CZ-16626 Prague, Czech Republic}

\date{Draft of \today}

\maketitle

\begin{abstract}
In the context of optimal control, we consider the inverse problem of Lagrangian identification given system dynamics and optimal trajectories. Many of its theoretical and practical aspects are still open. Potential applications are very broad as a reliable solution to the problem would provide a powerful modeling tool in many areas of experimental science. We propose to use the Hamilton-Jacobi-Bellman sufficient optimality conditions for the direct problem as a tool for analyzing the inverse problem and propose a general method that attempts at solving it numerically with techniques of polynomial optimization and linear matrix inequalities.
The relevance of the method is illustrated based on simulations on academic examples under various settings.
\end{abstract}

\section{INTRODUCTION}

In brief the {\it Inverse Optimal Control Problem} (IOCP) can be stated as follows.
Given a system dynamics
\[
\dot{x}(t)=f(x(t),u(t))
\]
with possibly state and/or control constraints
\[
x(t)\in X, \:\:u(t)\in U, \:\:t \in [0,T]
\]
and a set of trajectories
\[
(x(t;x_0),u(t;x_0))_{t\in[0,T],\:x_0 \in X}
\]
parametrized by time and initial states, and stored in a database,
the goal is to find a {\it Lagrangian} function
\[
l:X\times U \to\RR
\]
such that all state and control trajectories in the database are {\it optimal trajectories} for the direct Optimal Control Problem
(OCP) with integral  cost
\[
\int_0^Tl(x(t),u(t))dt,
\]
with fixed or free terminal time $T$.

Inverse problems in the context of calculus of variations and control are old topics that arouse a renewal of interest in the context of optimal control, especially in humanoid robotics. Actually, even the well-posedness of the IOCP is an issue and a matter of debate between the robotics and computer science communities.

\subsection{Context}

The problem of variational formulation of differential equations (or the inverse problem of the calculus of variations) dates back to the \textit{19-th} century. The one dimensional case is found in \cite{darboux1972leccons}, historical remarks are found in \cite{douglas1941solution,tonti1984variational}. Necessary and sufficient conditions for the existence and the enumeration of solutions to this problem have been investigated since, see also \cite{Saunders2010} for a survey about recent developments.
Notice that calculus of variations problems correspond to the particular choice of dynamics $f=u$ in the OCP.

Kalman first formulated the inverse problem in the context of linear quadratic regulator (LQR) \cite{kalman1964linear} which triggered research efforts in this realm \cite{anderson1971linear,jameson1973inverse,fujii1984complete}. Departing from the linear case, Hamilton-Jacobi theory is used in \cite{thau1967inverse} to recover quadratic value functions, in \cite{moylan1973nonlinear} to prove existence theorems for a class of inverse control problems, and in \cite{casti1980general} to generalize results obtained for LQR. In a slightly different context, \cite{freeman1996inverse} linked Lyapunov and optimal value functions for optimal stabilization problems. More recently the well-posedness issue was addressed in \cite{nori2004linear} in the context of LQR. Robustness and continuity aspects with respect to Lagrangian variations were investigated in \cite{chittaro2013inverse}, results about well-posedness and experimental requirements were exposed in \cite{ajami2013humans}, both in the context of Dubbins dynamics and strictly convex positive Lagrangians.

On a more practical side, motivated by \cite{arechavaleta2008optimality}, the authors in \cite{mombaur2010human} have proposed an algorithm
based on the ability to solve the direct problem for parametrized Lagrangians and on a
formulation of the IOCP as a (finite-dimensional) nonlinear program. Similar approaches have been proposed in the context of Markov Decision Processes \cite{abbeel2004apprenticeship,ratliff2006maximum}. In  a sense these methods are ``blind'' to the problem structure as
testing optimality at each current candidate solution is performed by solving numerically the associated direct OCP. To further
exploit the problem structure, we can use explicit analytical optimality conditions for the direct OCP. For instance the use of necessary optimality conditions 
for solving numerically the IOCP have already been proposed in recent works. 
In \cite{puydupin2012convex} the direct problem is first discretized and the IOCP
is expressed as a (finite-dimensional) nonlinear program. Then the 
{\it residual} technique of \cite{keshavarz2011imputing} for inverse parametric optimization is applied to the
Karush-Kuhn-Tucker optimality conditions of the nonlinear program to account for optimality of the database trajectories. On the other hand, \cite{hatz2012estimating} proposes to use the maximum principle for kinetic parameter estimation.

Surprisingly, in all the above references, only the seminal theoretical works of \cite{thau1967inverse,moylan1973nonlinear,casti1980general} are based on (or use) the Hamilton-Jacobi-Bellman equation (HJB in short), whereas the HJB provide a well-known sufficient condition for optimality and a perfect practical tool for verification. One reason might be that the HJB is rarely  used for solving the direct OCP and is rather used only as a verification tool.
\subsection{Contribution}
We claim and hope  to convince the reader that the HJB optimality equation (in fact even a certain relaxation of HJB) is a very appropriate tool not only for analyzing but also for  solving the IOCP. Indeed:

(a) The HJB optimality equation provides an almost perfect criterion to express optimality for database trajectories.

(b) The HJB optimality equation (or its relaxation) {\it sheds light}  on the many potential pitfalls related to the IOCP when treated in full generality\add{. Among them, the most concerning issue is the ill-posedness of the inverse problem and the existence of solutions carying very little physical meaning. The HJB condition} can be used as a guide to restrict the 
search space for candidate Lagrangians. 
\add{Previous approaches to deal with this problem, implicitly or explicitly, involve strong constraints on the class of functions among which candidate Lagrangians are searched \cite{mombaur2010human,puydupin2012convex,chittaro2013inverse,ajami2013humans}. This allows, in some cases, to alleviate the ill-posedness issue and to provide theoretical guarantees regarding the possibility to recover a true Lagrangian from the observation of optimal trajectories \cite{chittaro2013inverse,ajami2013humans}.}
\add{Departing from these approaches, the search space restrictions that we propose} stems from either simple relations between the 
candidate Lagrangian and optimal value function associated with the OCP, obvious from the HJB equation, or from purely sparsity biased estimation arguments. We do {\it not} require an {\it a priori} (and always questionable) selection of the ``type'' of candidate Lagrangian (e.g. coming from some ``physical" observations and/or remarks).

(c) Last but not least, a relaxation of the HJB optimality equation can be readily translated into a positivity condition for a certain function on some set. If the vector field $f$ is polynomial, and the state and/or control constraint sets $X$ and $U$ are basic semi-algebraic then a natural strategy is to consider {\it polynomials} as candidate Lagrangians and (approximated) optimal value functions associated with the direct OCP. Within this framework, using powerful positivity certificates {\it \`a la Putinar} to look for an optimal solution of the IOCP reduces to solving a hierarchy of semidefinite programs (SDP), or
linear matrix inequalities (LMI) of increasing size. 
Importantly, a distinguishing feature of this approach is {\it not} to rely on iteratively solving a direct OCP.
Notice that since the 1990s, the availability of reasonably efficient SDP solvers  \cite{vandenberghe1996semidefinite} has strengthened the interest in optimization problems with polynomial data. Examples of applications of such positivity certificates are given in \cite{lasserre2001global} 
and \cite{lasserre2008nonlinear} for direct optimization and optimal control, and in \cite{lasserre2013inverse} for inverse (static) optimization.

The paper is organized as follows. In Section \ref{sec:2} we properly define the IOCP. Next, in section \ref{sec:3}, we present the conceptual ideas based on
HJB theory, polynomial optimization and LMI. We emphasize that these ideas allow to highlight potential pitfalls of the IOCP. A practical implementation of the discussed conceptual method is proposed in section \ref{sec:4}. Finally, in section \ref{sec:5} we first illustrate 
the outputs of the method on a simple one-dimensional example 
and then provide promising results on more complex academic examples.
\section{PROBLEM FORMULATION}
\label{sec:2}
\subsection{Notations}
If $A$ is a topological vector space, $\mathcal{C}(A)$ represents the set of continuous functions from $A$ to $\RR$ and $\mathcal{C}^1(A)$ represents the set of continuously differentiable functions from $A$ to $\RR$. Let $X \subseteq \RR^{d_X}$ denote the state space and $U \subseteq \RR^{d_U}$ denote the control space which are supposed to be compact subsets of Euclidean spaces. System dynamics are represented by a continuously differentiable vector field
$f \in \mathcal{C}^1(X \times U)^{d_X}$
which is supposed to be known. Terminal state constraints are represented by a set $X_T \subseteq X$ which is also given.
Let $B_n$ denote the unit ball of the Euclidean norm in $\RR^n$, 
and let $\partial X$ denote the boundary of set $X$.

\subsection{Direct OCP}

Given a Lagrangian
\begin{align*}
	l_0 \in \mathcal{C}(X \times U),
\end{align*}
we consider a direct OCP of the form:
\begin{equation}\label{eq:pdirect}
	\tag{OCP}
\begin{array}{ll@{\;}l@{}}
	v_0(t,z) :=&  \displaystyle\inf_{u} & \displaystyle\int_t^T l_0(x(s), u(s)) ds \\
	& \mathrm{s.t.} & \dot{x}(s) = f(x(s), u(s)), \\
           && x(s) \in X,\,u(s) \in U,\, s \in [t,T],\\
	&&x(t) = z, \,x(T) \in X_T
\end{array}
\end{equation}
with final time $T \geq t$. In this formulation, the final time $T$ might be given or free, in which case it is a variable of the problem and the {\it value function} $v_0$ does not depend on $t$. For the rest of this paper, direct problems of the form of (\ref{eq:pdirect}) are denoted as
\begin{align*}
	\text{OCP}(l_0,z).
\end{align*}
\begin{remark}
	More general problem classes could be considered. Indeed, there is no terminal cost and the problem is stationary. Further comments are found in next sections.
\end{remark}
\begin{remark}
	The infimum in (\ref{eq:pdirect}) might not be attained.
\end{remark}

\subsection{IOCP}

The inverse problem consists of recovering the Lagrangian of the direct OCP based on:
\begin{itemize}
	\item knowledge of the dynamics $f$ as well as state and/or control constraint sets $X,U,X_T$, 
	\item observation of optimal trajectories stored in a database indexed by some index set $I$ \end{itemize}
\begin{align*}
	\mathcal{D} = \{(t_i, x_i, u_i)\}_{i \in I} \in ([0,T] \times X \times U)^I.
\end{align*}
\begin{remark}
	For complete trajectories, the index set $I$ is a continuum. For real world data, it is  a finite set.
\end{remark}
More specifically, the inverse problem consists of finding ${l} \in \mathcal{C}(X \times  U)$ such that $\mathcal{D}$ is a subset of optimal trajectories of problem $\text{OCP}({l}, x_i(t))$ for all $t \in [0,T]$. A solution of this problem would be an operator
mapping back $\mathcal{D}$ to $l$
such that the input data satisfy optimality conditions for problem $\text{OCP}({l},x_i(t))$. 

\section{HAMILTON-JACOBI-BELLMAN FOR THE INVERSE PROBLEM}
\label{sec:3}

\subsection{Hamilton-Jacobi-Bellman theory}

For the rest of this paper, we define the following
linear operator acting on Lagrangians and value functions:
\[
\mathcal{L} \colon (l, v) \to l + \frac{\partial v}{\partial t} + \frac{\partial v}{\partial x}^T f.
\]
A well-known sufficient condition for optimality in problem (\ref{eq:pdirect}) follows from
Hamilton-Jacobi-Bellman (HJB) theory, see e.g. \cite{athans1966optimal,bardi2008optimal}.

\begin{proposition}\label{prop-1}
Suppose that there exist state and control trajectories $(x_0(s),u_0(s))_{s\in[t,T]}$ such that
\begin{align*}\label{eq:dynCstr}\tag{A1}
\dot{x}_0(s)= f(x_0(s),u_0(s)), \\
x_0(s) \in X,\: u_0(s) \in U,\, s \in [t, T],\\
 x_0(t) = z, \:\:x_0(T) \in X_T.
\end{align*}
Suppose in addition that there exists a function $v_0 \in \mathcal{C}^1([t,T] \times X )$ such that
\begin{align}
\label{aux1} \mathcal{L}(l_0,v_0)(s,x,u)\geq  0,\, \forall (s,x,u)\in [t,T]\times X\times U,\\ 
\label{aux3} v_0(T,x) = 0,\,\forall x \in X_T,\\
\label{aux2} \mathcal{L}(l_0,v_0)(s,x_0(s), u_0(s)) = 0,\, \forall s \in [t, T].
\end{align}
Then the state and control trajectories $(x_0(s),u_0(s))_{s\in[t,T]}$ are optimal solutions of the direct problem (\ref{eq:pdirect}).
\end{proposition}

	\begin{proof}
	Recall that $X,U$ are compact. So integrating (\ref{aux1}) along {\it any} feasible state and control
trajectories $(x(s),u(s))$ with initial state $z$ at time $t$, and using (\ref{aux3}) yields $\int_t^Tl_0(x(s),u(s))ds\geq v_0(t,z)$ whereas using (\ref{aux2}) one has
$\int_t^Tl_0(x_0(s),u_0(s))ds= v_0(t,z)$. \end{proof}

	Observe that (\ref{aux1})-(\ref{aux3}) are just a {\it relaxation} of the HJB equation and from Proposition \ref{prop-1}
(\ref{aux1})-(\ref{aux2}) provide  a certificate of optimality of the proposed trajectory. Additional assumptions on problem structure are required to make this condition necessary, see \textit{e.g.} \cite{athans1966optimal}. Moreover, for most direct problems these conditions can not be met in the usual sense and {\it viscosity solutions} are needed, see \textit{e.g.} \cite{bardi2008optimal}. Our approach is based on a classical interpretation as well as a relaxation of these sufficient conditions alone.

\begin{remark}
In the free terminal time setting, the conditions (\ref{aux1})-(\ref{aux3}) can be simplified because $v_0$ does not depend on the initial time $t$ any more.
\end{remark}

\subsection{Inverse problem: main idea}
We use HJB relations to characterize some approximate solutions of the inverse problem. \add{The idea is based on the following
weakening of Proposition \ref{prop-1}.
\begin{proposition}
\label{prop-2}
 Let $(x_0(s), u_0(s))_{s \in [t, T]}$, with  $x_0(t)=z$,  be such that (\ref{eq:dynCstr}) is satisfied. Suppose that 
 there exist a real $\epsilon$ and functions $l \in \mathcal{C}(X\times U)$, $v \in \mathcal{C}^1([0,T]\times X)$ such that
\begin{align}
\mathcal{L}({l},{v})(s,x,u) \geq 0,\, \forall(s,x,u)\in[t,T]\times X\times U,  \label{eq:conicCstr}\\
v(T,x)=0,\, \forall x \in X_T,\\
\int_t^T \mathcal{L}({l},{v})(s,x_0(s),u_0(s)) ds \leq \epsilon. \label{eq:fitCstr}
\end{align}
Then, the input trajectory $(x_0(s), u_0(s))_{s \in [t, T]}$ is $\epsilon$-optimal for problem $\text{OCP}({l},z)$. 
\end{proposition}
\begin{proof}
Proceeding as in the proof of Proposition \ref{prop-1}, ${v}(t,z)$ is a lower bound on the value function of problem $\text{OCP}({l},z)$ and from the last linear constraint (\ref{eq:fitCstr}) we deduce that
\begin{align*}
	\int_t^T {l}(x_0(s), u_0(s))ds \leq {v}(t,z) + \epsilon.
\end{align*}
\end{proof}
Proposition \ref{prop-2} can be easily extended to the case of multiple trajectories.} The main advantage is that we have a certificate of $\epsilon$-optimality. Observe that $0$-optimality is in principle impossible to obtain because the optimal value function $v$ of a direct OCP is in general non-differentiable; however as soon as $v$ is continuous then by the Stone-Weierstrass Theorem $v$ can be approximated on the compact $[t,T]\times X$ as closely as desired by a polynomial and so $\epsilon$-optimality is indeed achievable for arbitrary $\epsilon>0$.
\begin{remark}
	\add{The conditions (\ref{eq:conicCstr})-(\ref{eq:fitCstr})} do not ensure that the infimum of the direct problem
	\add{with cost ${l}(x,u)$} is attained. However the fact that a trajectory is given ensures that it is feasible.
\end{remark}
\subsection{Multiple and trivial solutions}
\add{The construction of multiple solutions to the inverse problem is trivial given the tools of Proposition \ref{prop-2}. Consider the free terminal time setting and suppose that the pair $({l}, {v})$ satisfies conditions (\ref{eq:conicCstr})-(\ref{eq:fitCstr}) for some $\epsilon > 0$. Take a differentiable $\tilde{v}$ such that $\tilde{v}(T,x) = 0$ for $x \in X_T$. Then the pair $({l} - f^T {v}, {v} + \tilde{v})$ satisfies constraints (\ref{eq:conicCstr})-(\ref{eq:fitCstr}) for the same $\epsilon$. The possibility to construct such solutions stems from the existence of trivial solutions to the problem.}

As we already mentioned, the constraint (\ref{eq:conicCstr}) is \add{positively} homogeneous in $({l},{v})$ and therefore, \add{the pair $(0,0)$ is always feasible}. \add{In other words the trivial cost ${l} = 0$ is always an optimal} solution of the inverse problem, independently of input trajectories. \add{But the well-posedness issue} is even worse than this. Consider a pair of functions $({l},{v})$ such that $\mathcal{L}({l}, {v}) = 0$ on the domain of interest, then any feasible trajectory of the direct problem will be optimal for ${l}$. 
\begin{example}
	Consider the following one dimensional free terminal time direct setting
	\begin{align*}
		X = U = B_1, \,X_T = \{0\}, \, f(x, u) = u.
	\end{align*}
The pair of functions $({l}, {v}) = (-2xu, x^2)$ satisfies constraint (\ref{eq:conicCstr}) and $\mathcal{L}({l}, {v}) = 0$. Any feasible trajectory, $(x_0(s), u_0(s))_{s \in [t, T]}$ with $t < T$, $x_0(t)=z$, and such that (\ref{eq:dynCstr}) is satisfied, is optimal for $\text{OCP}({l},z)$. Indeed
	\begin{align*}
		\int_t^T {l}(x_0(s), u_0(s))ds = \int_t^T \frac{d}{dt}(-x_0(s)^2)ds = x_0(t)^2.
	\end{align*}
\end{example}
Such pairs are solutions of the IOCP in the sense that was proposed in the previous section. However, these solutions do not have any physical interpretation, because they do not depend on input trajectories. In addition, because the solutions of the IOCP form a {\it convex cone}, \add{the existence of such solutions allows to construct multiple solutions to the IOPC}. To avoid this, one possibility is to include an additional normalizing constraint of the form:
\begin{align}
	\mathcal{A}(\mathcal{L}({l}, {v})) = 1
	\label{eq:normCstr}
\end{align}
for some linear functional $\mathcal A$ on the space of continuous functions. This can be viewed as a search space reduction as we intersect the {cone} defined by (\ref{eq:conicCstr}) with an affine space. 
\begin{remark}
	As we already mentioned, we do not consider terminal cost and limit ourselves to stationary problems. This setting was chosen to avoid more ill-posedness of the same type and keep the presentation clear.
\end{remark}
\begin{remark}
	\add{Previous practical methods \cite{mombaur2010human,puydupin2012convex} and theoretical work \cite{nori2004linear,chittaro2013inverse,ajami2013humans} include, implicitly or explicitly, constraints of the form of (\ref{eq:normCstr}) but only enforce them on the candidate Lagrangian ${l}$.}
\end{remark}

\subsection{Considering multiple trajectories}

Considering a single trajectory as input for the IOCP leads to Lagrangians that enforce closedness to this trajectory, which may have little physical meaning.
\begin{example}
	Consider the direct problem with free terminal time
	\begin{align*}
		X = U = B_2, \,X_T = \{0\}, \, f(x, u) = u, \, l(x,u) = 1.
	\end{align*}
	The optimal control  is $u(x) = \frac{-x}{||x||_2}$ and the optimal value function is $v(x) = ||x||_2$. Consider the trajectory
	\begin{align*}
		(x_0(s), u_0(s))_{s \in[0, 1]} = ((s - 1, 0), (1, 0))_{s \in[0, 1]}.
	\end{align*}
This trajectory is optimal for $\text{OCP}(l,(-1,0))$ but it is also optimal for $\text{OCP}(||u - (1,0)||_2^2,(-1,0))$. However, the second Lagrangian only captures a constant of the particular trajectory considered.
\end{example}

\subsection{Effect of discretization}

In practical settings, one rarely has access to complete trajectories. Typical experiments produce discrete samples from trajectories and possibly with additional experimental noise. The database consists of a set of $n \in \NN$ points $\{(t_i,x_i, u_i)\}_{i = 1 \ldots n}$. In this case, we replace the integral in (\ref{eq:fitCstr}) by a discrete sum:
\begin{align}
	\frac{1}{n}\sum_{i=1}^n \mathcal{L}({l},{v})(t_i,x_i,u_i) \leq \epsilon.
	\label{eq:fitDiscCstr}
\end{align}
In this setting, it is generically possible to find Lagrangians that satisfy constraints (\ref{eq:conicCstr}) and (\ref{eq:fitCstr}) with $\epsilon = 0$.
\begin{example}
	Consider, in a free terminal time setting a discrete dataset $\{(x_i, u_i)\}_{i=1 \ldots n}$. The pair of functions $\left((x, u) \to \prod_i||x_i -x||^2 ||u_i - u||^2, 0\right)$ satisfies constraint (\ref{eq:conicCstr}) and (\ref{eq:fitDiscCstr}) with $\epsilon = 0$.
\end{example}
Because of experimental noise and random perturbations of the input trajectories, one can find Lagrangians which fit tightly a particular sample of trajectories. This does not give much insight about the true nature of the original trajectories. Furthermore, the proposed Lagrangian may be far from optimal if the database was fed up with a different sample of trajectories.  In statistics this well-known phenomenon bears the name of \textit{overfitting} (see \textit{e.g.} \cite{babyak2004you} for a nontechnical introduction and \cite{vapnik1999overview} for an overview of the mechanisms it involves). A common approach to avoid overfitting is to introduce biases in the estimation procedure with search space restrictions or regularization terms. We adopt such a strategy in the practical method described in the next section.
\begin{remark}
	It is clear that this discretization, although intuitive, arises many questions regarding the effect of noise and of sample size in practical IOCP, rarely discussed and even mentioned in previous works. These theoretical considerations are not specific to the method we propose. They are beyond the scope of this paper and  constitute a strong motivation for future research work.
\end{remark}

\section{A PRACTICAL IMPLEMENTATION}
\label{sec:4}

We have seen that the HJB theory is very useful to address well-posedness issues regarding the inverse problem. In full generality the HJB equations (or their relaxations) are computationally intractable. On the other hand, in a polynomial and semi-algebraic context, some powerful positivity certificates from real algebraic geometry permit to translate the relaxation (\ref{eq:conicCstr}) of the HJB equations into an appropriate LMI hierarchy, hence amenable to practical computation.
Therefore, in the sequel, we assume that $f$ is a polynomial and $X$, $U$ and $X_T$ are basic semi-algebraic sets. Recall that $G \subset \RR^d$ is basic semi-algebraic whenever there exists polynomials $\{g_i\}_{i=1 \ldots m}$ such that 
$G = \left\{ x:\, g_i(x) \geq 0, i = 1 \ldots m \right\}$.
In addition, we assume that the input database  is indexed by a finite set:
$\mathcal{D} = \{(t_i, x_i, u_i)\}_{i=1,\ldots,n}$.

\subsection{Problem formulation}
We consider the following program
\begin{equation}\tag{IOCP}\label{eq:primal1}
\begin{array}{l@{\;}l}
\displaystyle \inf_{l, v, \epsilon} &  \epsilon + \lambda ||l||_1\\
\mathrm{s.t.}	& \mathcal{L}(l, v)(t,x,u) \geq 0, \:\:\forall (t,x,u)\in[0,T]\times X\times U, \\
& v(T,x) = 0, \:\:\forall x \in X_T, \\
& \frac{1}{n} \displaystyle\sum_{i=1}^n\mathcal{L}(l, v)(t_i,x_i,u_i) \leq \epsilon, \\
& \mathcal{A}(\mathcal{L}(l, v)) = 1
\end{array}
\end{equation}
where $l$ and $v$ are polynomials, $\epsilon$ is a real, $\lambda>0$ is a given regularization parameter, and $||.||_1$ denotes the $\ell_1$ norm of a polynomial, \textit{i.e.} the sum of absolute values of its coefficients when expanded in the monomial basis. The first two constraints come from the \add{relaxation (\ref{eq:conicCstr}) of the HJB equations while the third constraint comes from the {\it fit} constraint (\ref{eq:fitDiscCstr}). Finally, the last affine constraint is meant to avoid the trivial solutions that satisfy  $\mathcal{L}(l, v) = 0$}. The $\ell_1$ norm is not differentiable around sparse vectors (with entries equal to zero) and  has therefore a sparsity promoting role which allows to bias \add{solutions of the problem toward Lagrangians with  few nonzero coefficients. This regularization affects the problem well-posedness and will prove to be essential in numerical experiments.}
\subsection{Implementation details}
Linear constraints are easily expressed in term of polynomial coefficients. A classical lifting allows to express the $\ell_1$ norm minimization as a linear program. The normalization \add{functional $\theta$ is chosen}  to be an integral over a box contained in $S$. \add{To express nonnegativity of polynomials over a compact basic semi-algebraic set of the form
$G = \{x: \,g_i(x) \geq 0,\:i=1,\ldots,m\}$,
 we invoke powerful positivity certificates from real algebraic geometry. Indeed, 
 if a polynomial $p$ is positive on $G$ then by Putinar's Positivstellensatz \cite{putinar1993positive} it can be written as 
\begin{align}
	\label{eq:sosCstr}
	p &= p_0 + \sum_{i=1}^m g_i p_i, \quad p_i \in \Sigma^2,\quad i=1,\ldots,m
\end{align}
where $\Sigma^2$ denotes the set of sum of squares (SOS) polynomials. Hence (\ref{eq:sosCstr})  
provides a useful certificate that $p$ is nonnegative on $G$. Moreover, as membership to $\Sigma^2$ reduces to semidefinite programming,
the constraint (\ref{eq:sosCstr}) is easily expressed as an LMI
whose size depends on the degree bound allowed for the SOS polynomials $p_i$ in (\ref{eq:sosCstr}). Therefore, replacing the positivity constraint in (\ref{eq:primal1}) with the constraint (\ref{eq:sosCstr}) allows to express problem (\ref{eq:primal1}) as a hierarchy of LMI problems \cite{vandenberghe1996semidefinite} indexed by the degree bounds on the SOS in (\ref{eq:sosCstr}).
Thus each LMI of the hierarchy can be solved efficiently (of course up to some size limitations)}. We use the SOS module  of the \texttt{YALMIP} toolbox \cite{lofberg2009pre} to manipulate and express polynomial constraints at a high level in \texttt{MATLAB}.

\section{NUMERICAL EXPERIMENTS}
\label{sec:5}
\subsection{General setting}
In our numerical experiments we considered several direct problems of the same form as (\ref{eq:pdirect}). That is, we give ourselves compact sets $X$, $U$, $X_T$, the dynamics $f$, and a Lagrangian $l_0$. We take known examples for which the optimal control law can be computed. Given these, we generate randomly $n$ data points $\mathcal{D} = \{(t_i,x_i, u_i)\}_{i=1 \ldots n}$ in the domain, such that $u_i$ is the optimal control value at point $x_i$ and time $t_i$. For a given value of $\lambda$, we compute a solution ${l}$ of problem (\ref{eq:primal1}). We can then measure how ${l}$ is close to $l_0$. Note that in some simulations, the control is corrupted with noise.
\subsection{Benchmark direct problems}
\subsubsection{Minimum exit time in dimension 1}
\label{sec:pb1}
For this problem we take
\begin{align*}
	X = U =B_1, \, X_T = \partial X, \, l_0= 1, \, f = u.	
\end{align*}
The optimal law for this problem is $u = \text{sign}(x)$ and the value function is $v_0(x) = 1 - |x|$.
\subsubsection{Minimum exit time in dimension 2}
\label{sec:pb2}
For this problem we take
\begin{align*}
	X = U = B_2, \, X_T = \partial X, \, l_0= 1, \, f = u.	
\end{align*}
The optimal law for this problem is $u = \frac{x}{||x||_2}$ and the value function is $v_0(x) = 1 - ||x||_2$.
\subsubsection{Minimum exit norm in dimension 2}
\label{sec:pb3}
For this problem we take
\begin{align*}
	X = U = B_2, \, X_T =  \partial X, \\
	l_0= ||x||_2^2 + ||u||_2^2, \, f = u.	
\end{align*}
The optimal law for this problem is $u = x$ and the value function is $v_0(x) = 1 - ||x||_2^2$.
\subsubsection{Fixed time double integrator with quadratic cost}
\label{sec:pb4}
For this problem we take
\begin{align*}
	X &= B_2, \, U = [-r,r], \, T=1,\\
	l_0&= x^T \left(  
	\begin{array}{cc}
		2 &\frac{1}{4}\\
		\frac{1}{4} &1
	\end{array}
	\right)x + u^2,\\
	f &= 
	\left(  \begin{array}{c c}
		0 & 1\\
		0 & 0
	\end{array} \right) x
	+ 
	\left(  \begin{array}{c}
		0 \\
		1 
	\end{array} \right) u\\
\end{align*}
for a big enough value of $r$. This is an LQR problem, the optimal control is of the form $u(t) = -K(t)x(t)$ where $K(t)$ is obtained by solving the corresponding Riccati differential equation.

\subsection{Numerical results}

\begin{figure}[h!]
	\centering
	\includegraphics[width=0.8\textwidth]{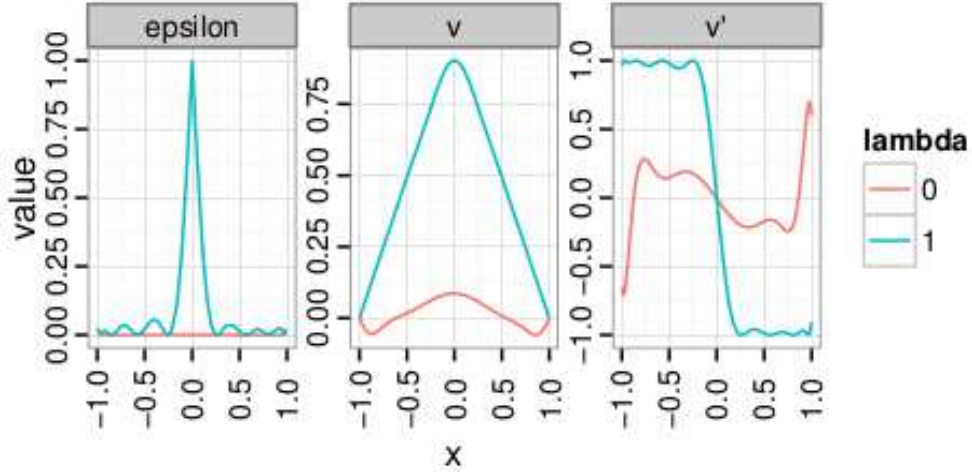}
	\caption{Solution for the one dimensional minimum exit time problem, effect of the regularization parameter $\lambda$. The first column is the distribution of the error $\epsilon$, the second is a representation of the value function ${v}$ and the third column is a representation of its derivative for solutions of problem (\ref{eq:primal1}) with and without regularization. We take 100 points on the segment. Lagrangian ${l}$ and  value function ${v}$ are both polynomials of degree 16.}
	\label{fig:oneDexit}
\end{figure}

\subsubsection{Illustration on a one dimensonal problem}
We consider the one dimensional minimum exit time problem in \ref{sec:pb1}. The results are presented in Figure \ref{fig:oneDexit}. We compare the output of (\ref{eq:primal1}) with ($\lambda = 1$) and without ($\lambda = 0)$ regularization. The main comments are
as follows.
\begin{itemize}
	\item Given any symmetric differentiable concave function ${v}$ vanishing on $\{-1, 1\}$, the pair $({l} = |{v}'|, {v})$ solves problem (\ref{eq:primal1}) with $\epsilon = 0$. 
	\item  Given any polynomial ${v}$ vanishing on $\{-1, 1\}$, and any positive polynomial $p$ on $[0,T]\times X\times U$,
the pair $({l} = (1 - u^2) p - u {v}', {v})$ solves problem (\ref{eq:primal1}) with $\epsilon = 0$. Note that this solution only captures the fact that $|u| = 1$.
	\item Any convex combination of solutions of the types mentioned above also solve problem (\ref{eq:primal1}). It is therefore very hard, in the absence of regularization ($\lambda = 0$), to recover the true Lagrangian.
	\item The sparsity inducing effect of $\ell_1$-norm regularization allows to recover the true Lagrangian ($\lambda = 1$).
	\item The value function associated to this Lagrangian is not smooth around the origin and therefore hard to approximate, hence the value of the error is high.
\end{itemize}
\begin{figure}[t]
	\centering
	\includegraphics[width=0.8\textwidth]{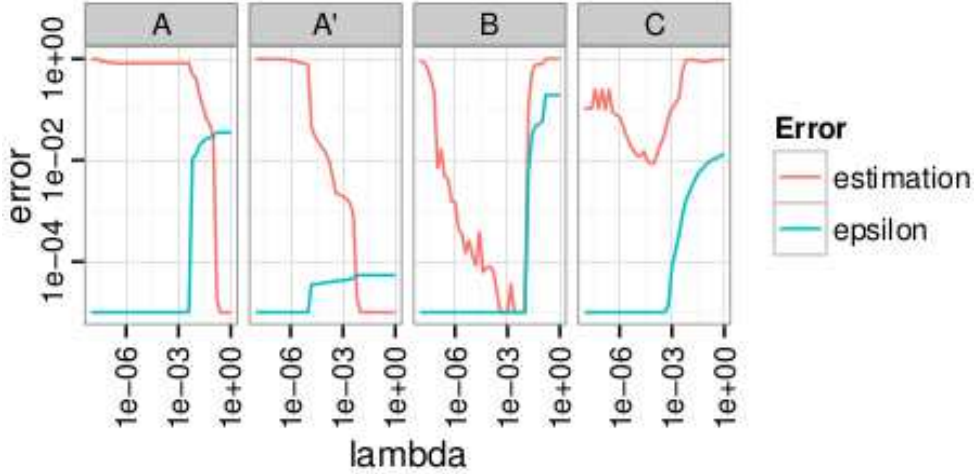}
	\caption{Deterministic setting in dimension 2, error versus variation of regularization parameter $\lambda$. A: minimum exit time. A': minimum exit time sampling far from the origin. B: minimum exit norm. C: double integrator. Estimation error is the metric defined in (\ref{eq:metric}). Epsilon error is the value of the term $\epsilon$ in program (\ref{eq:primal1}). For A, A' and B, we take 20 random points and corresponding control value. For C, we take 50 random points and time with corresponding control. For all problems, Lagrangian ${l}$ is a polynomial of degree 4 and value function ${v}$ is a polynomial of degree 10.}
	\label{fig:deterministic}
\end{figure}
\subsubsection{Lagrangian estimation in dimension 2}
We consider the following settings
\begin{itemize}
	\item[A] Problem \ref{sec:pb2} with $\{x_i\}$ sampled on $B_2$.
	\item[A'] Problem \ref{sec:pb2} with $\{x_i\}$ sampled on $B_2 \setminus \frac{1}{2}B_2$.
	\item[B] Problem \ref{sec:pb3} with $\{x_i\}$ sampled on $B_2$.
	\item[C] Problem \ref{sec:pb4} with $\{x_i\}$ sampled uniformly on the unit $\ell_\infty$ ball.
\end{itemize}
In all cases, ${l}$ has degree $4$ and ${v}$ has degree $10$.
\paragraph{Estimation error}
Since the inverse problem is positively homogeneous, our objective is to recover a Lagrangian $l_0$, up to a positive multiplicative factor. Since we use polynomials, the Lagrangians we estimate can be represented as vectors in the monomial basis. We use the following metric to account for the estimation error:
\begin{align}
	\inf_\alpha \frac{||l_0 - \alpha {l}||_2}{||l_0||_2}=\left(1 - \frac{\left\langle l_0, {l} \right\rangle^2}{||l_0||_2^2||{l}||_2^2}\right)^\frac{1}{2}
	\label{eq:metric}
\end{align}
where the scalar product of polynomials is defined as the usual scalar product of their vectors of coefficients
in the monomial basis.
\begin{figure}[t]
	\centering
	\includegraphics[width=0.8\textwidth]{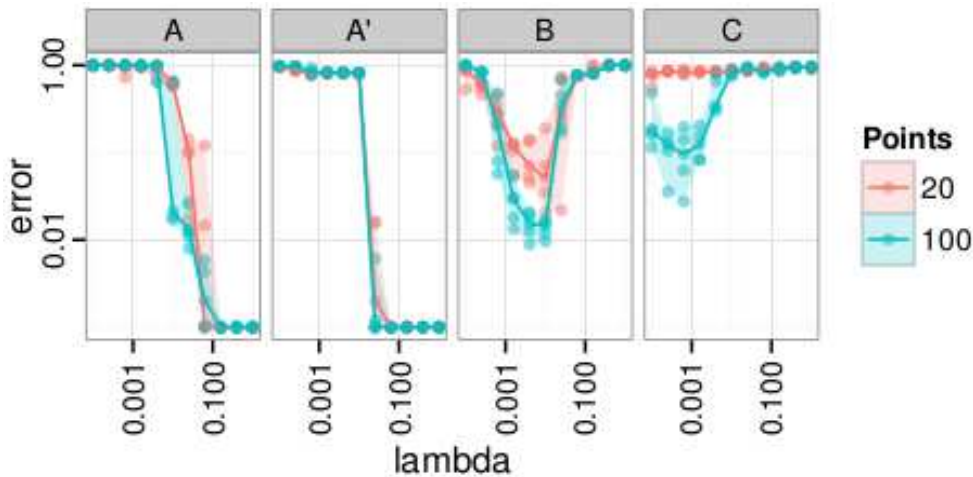}
	\caption{Stochastic setting in dimension 2, estimation error versus variation of regularization parameter $\lambda$. A: minimum exit time. A': minimum exit time sampling far from the origin. B: minimum exit norm. C: double integrator. Estimation error is the metric defined in (\ref{eq:metric}). We vary the number of random points. The control is corrupted by uniform noise and we repeat the experiment five times for each value of $\lambda$. For all problems, Lagrangian ${l}$ is a polynomial of degree 4 and value function ${v}$ is a polynomial of degree 10.}
	\label{fig:noise}
\end{figure}
\paragraph{Deterministic setting}
The results for the four problems are presented in Figure \ref{fig:deterministic}. For all problems, ${l}$ is of degree 4. Therefore, for problems A, A' and B, ${l}$ is represented by a 70-dimensional vector and for problem $C$, it is a 35-dimensional vector. When the estimation error is close to 1, we estimate a Lagrangian ${l}$ that is orthogonal to $l_0$ in the monomial basis. For all problems we are able to recover the true Lagrangian with good accuracy for some value of the regularization parameter $\lambda$. In the absence of regularization, we do not recover the true Lagrangian at all. This highlights the important role of $\ell_1$ regularization which allows to bias the estimation toward {\it sparse polynomials}. When the estimation error is minimal, the value of $\epsilon$ is reasonably low, depending on how the value function can be approximated by a polynomial. For example, A' shows lower $\epsilon$ value because we avoid sampling database points close to the nondifferentiable point of the true value function. In example C, the value function is not a polynomial and therefore harder to approximate. The estimation accuracy is still very reasonable.
\paragraph{Stochastic setting}
The results for the four problems are presented in Figure \ref{fig:noise}. The setting is similar to the deterministic case of the previous paragraph, except that we add uniform noise (of maximum magnitude $10^{-1}$) to the control input. Therefore, the problem is harder than the one presented in the previous paragraph. Several samples and noise realizations are considered to highlight the global trends. These simulations show that despite the stochastic corruption of the input control, we are still able to recover the true Lagrangian with reasonable accuracy. As one could expect, increasing the number of datapoints allows to recover the true Lagrangian with a better accuracy.

\section{CONCLUSIONS AND FUTURE WORKS}

We have presented how Hamilton-Jacobi-Bellman sufficient condition can be used to analyse the inverse problem of optimal control and proposed a practical method based on polynomial optimization and linear matrix inequality hierarchies to provide a candidate solution to this problem. Numerical results suggest that the method is able to estimate accurately Lagrangians comming from various optimal control problems. 

For the specific examples proposed, the optimality conditions allow to highlight many sources of ill-posedness for the inverse problem. In addition to a relaxation of the optimality conditions, we added a constraint and a penalization to circumvent ill-posedness. Numerical simulations support the idea that these are essential to estimate a Lagrangian accuartely. We do not rely on strong bias toward specific candidate Lagrangians and are able to perform accurate estimation in many different settings using the same regularization technique. A natural question that arises here is to find necessary conditions under which this is possible. 

The motivation for the use of Hamilton-Jacobi-Bellman optimality condition is the guarantees it provides when one has access to complete noiseless trajectories. However, practical experimental settings require to consider the effects of discretization and additive noise.
Along these lines, consistency and asymptotic properties of the proposed estimation procedure with respect to random discretization and noise are natural questions.

Finally, it is necessary to carry out further experiments on real world datasets in order to determine if the proposed method
works on practical inverse problems, our primary target being those coming from humanoid robotics
\cite{arechavaleta2008optimality,mombaur2010human}.

\section*{ACKNOWLEDGMENTS}

This work was partly funded by an award of the Simone and Cino del Duca foundation of Institut de France.
The authors would like to thank Fr\'ed\'eric Jean, Jean-Paul Laumond, Nicolas Mansard and Ulysse Serres for fruitful discussions.



\end{document}